\documentclass[a4paper,12pt]{article}
\usepackage[utf8]{inputenc}
\usepackage{amsmath,amsthm,amsfonts,latexsym,amssymb,bm,cite}
\usepackage{cite}

\newtheorem{Thm}{Theorem}[section]

\newtheorem{Lem}[Thm]{Lemma}

\newtheorem{Prop}[Thm]{Proposition}

\newtheorem{Def}[Thm]{Definition}

\newenvironment{altproof}[1]   
{\noindent   
{\em Proof of {#1}}.}   
{\nopagebreak\mbox{}\hfill $\Box$\par\addvspace{0.5cm}}

\newcommand{\R}{\mathbb{R}}
\newcommand{\ut}{\tilde{u}} 
\newcommand{\utp}{\tilde{u}^+}
\newcommand{\utm}{\tilde{u}^-}
\newcommand{\uh}{\hat{u}}
\newcommand{\uhp}{\hat{u}^+}
\newcommand{\uhm}{\hat{u}^-}
\newcommand{\ub}{\bar{u}}
\newcommand{\bu}{\underline{u}}
\newcommand{\vt}{\tilde{v}} 
\newcommand{\vtp}{\tilde{v}^+}
\newcommand{\vtm}{\tilde{v}^-}
\newcommand{\vh}{\hat{v}}
\newcommand{\vhp}{\hat{v}^+}

\newcommand{\utn}{\tilde{u}_n} 
\newcommand{\utpn}{\tilde{u}^+_n}
\newcommand{\utmn}{\tilde{u}^-_n}
\newcommand{\uhn}{\hat{u}_n}
\newcommand{\uhpn}{\hat{u}^+_n}
\newcommand{\uhmn}{\hat{u}^-_n}
\newcommand{\ubn}{\bar{u}_n}
\newcommand{\bun}{\underline{u}_n}

\newcommand{\whp}{\hat{w}^+}

\newcommand{\ot}{\tilde{\omega}}
\newcommand{\ohat}{\hat{\omega}}
\newcommand{\ob}{\bar{\omega}}
\newcommand{\rt}{\tilde{r}}
\newcommand{\st}{\tilde{s}}
\newcommand{\that}{\hat{t}}
\newcommand{\rtn}{\tilde{r}_n}
\newcommand{\stn}{\tilde{s}_n}
\newcommand{\thn}{\hat{t}_n}
\newcommand{\hozo}{H^1_0(\Omega)}
\newcommand{\hoz}{\underline{H}}
\newcommand{\iu}{I_\mu}
\newcommand{\ap}{a^+}
\newcommand{\am}{a^-}
\newcommand{\nmu}{{\cal{N}}_\mu}

\newcommand{\up}{\frac{\mu}{p}}
\newcommand{\op}{\frac{1}{p}}
\newcommand{\oh}{\frac{1}{2}}

\begin{document}
\begin{center}
{\Large\bf 
Multibump nodal solutions for an indefinite superlinear elliptic problem}\\ 
\ \\
Pedro M.\ Girão\footnote{Email: pgirao@math.ist.utl.pt. Partially supported by the Center for Mathematical
Analysis, Geometry and Dynamical Systems through FCT Program POCTI/FEDER and
by grant POCI/FEDER/MAT/55745/2004.} and José Maria Gomes\footnote{Email:
jgomes@math.ist.utl.pt. Supported by FCT
grant SFRH/BPD/29098/2006.}
\\
Instituto Superior Técnico\\
Av.\ Rovisco Pais\\
1049-001 Lisbon, Portugal
\end{center}

\begin{center}
{\bf Abstract}\\
\end{center}
\noindent We define some Nehari-type constraints 
using an orthogonal decomposition of the Sobolev space $H^1_0$
and prove the existence 
of multibump nodal solutions for an indefinite superlinear elliptic problem.

\section{Introduction}
Consider a Lipschitz bounded domain $\Omega\subset\R^N$, $N\geq 1$, and a function
$a\in C(\bar{\Omega})$, with $a=\ap-\am$, where $\ap=\max\{a,0\}$ as usual.
Assume the set $\ap>0$ is the union of a finite number, $L\geq 1$, of open connected and
disjoint Lipschitz components. We separate the components arbitrarily into three families 
\begin{eqnarray*}
\Omega^+=\{x\in\Omega\::\:\ap(x)>0\}&=&\displaystyle
\left(\cup_{i=1}^I\tilde{\omega}_i\right)\cup\left(\cup_{j=1}^J\hat{\omega}_j\right)\cup\left(\cup_{k=1}^K\bar{\omega}_k\right)\nonumber\\
&=&\tilde{\Omega}\cup\hat{\Omega}\cup\underline{\Omega},\nonumber\\
\noalign{\noindent so that $L=I+J+K$; we also assume}\nonumber\\
\Omega^-=\{x\in\Omega\::\:\am(x)>0\}&=&\displaystyle\Omega\setminus\overline{\Omega^+}.\nonumber
\end{eqnarray*}
Let $\mu>0$  and $p$ be a super\-quadratic
and subcritical exponent, $2<p<2^*$, with $2^*=2N/(N-2)$ for $N\geq 3$, and
$2^*=+\infty$ for $N=1$ or 2. Our main result is
\begin{Thm}\label{zero}
For every large $\mu$, there exists an $\hozo$ weak solution $u_\mu$ of
\begin{equation}
\label{one-one}
-\Delta u=(\ap-\mu\am)|u|^{p-2}u\quad\mbox{in}\ \Omega.
\end{equation}
Furthermore, the family $\{u_\mu\}$ has the property that (modulo a subsequence)
\begin{equation}
\label{twenty-three}
u_\mu\rightharpoonup u\quad\mbox{in}\ \hozo\ \mbox{as}\ \mu\to+\infty,
\end{equation}
where 
$$
\left\{\begin{array}{l}
-\Delta u=\ap|u|^{p-2}u\quad\mbox{in}\ \ot_i,\\
u^\pm\not\equiv 0\quad\mbox{in}\ \ot_i,
\end{array}\right.\qquad i=1,\ldots,I,
$$
$$
\left\{\begin{array}{l}
-\Delta u=\ap|u|^{p-2}u\quad\mbox{in}\ \ohat_j,\\
u^+\not\equiv 0,\ u^-\equiv 0\quad\mbox{in}\ \ohat_j,
\end{array}\right.\qquad j=1,\ldots,J,
$$
$$
u\equiv 0\quad\mbox{in}\ \ob_k,\qquad k=1,\ldots,K,
$$
and
$$
u\equiv 0\quad\mbox{in}\ \Omega^-. 
$$
\end{Thm}

The one-dimensional version of (\ref{one-one}) was studied in
\cite{GHZ} with topological shooting arguments and
phase-plane analysis.
Theorem~\ref{zero} extends the main result in \cite{BGH} where the case $\tilde{\Omega}=\emptyset$ was considered,
so that the function $u$ in (\ref{twenty-three}) was positive.
The authors used a volume constrain regarding the $L^p$ norm, rescaling
and a min-max argument based on the Mountain Pass Lemma. A careful analysis allowed them to distinguish
between the solutions that arise from 
the $2^L$ different possible partitionings of
$\Omega^+=\hat{\Omega}\cup\underline{\Omega}$. However,
the argument in \cite{BGH} does not seem either to extend easily to the present situation or
to be suited to non-homogeneous nonlinearities.

Our approach is adapted from the work \cite{RT}
regarding a system of equations related to
$$
\left\{\begin{array}{l}
-\epsilon^2\Delta u+V(x)u=f(u)\quad\mbox{in}\ \Omega,\\
u>0\quad\mbox{in}\ \Omega,
\end{array}\right.
$$
when $\epsilon$ is small and the functions $V$ and $f$ satisfy
appropriate conditions. 
The positive function $V$ was assumed to have a finite number of minima.
In particular,
the authors proved the existence of multipeak positive solutions 
by defining a Nehari-type manifold which, roughly speaking,
imposes that the derivative of the associated Euler-Lagrange functional
at a function $u$ should vanish when applied to a truncation
of $u$ around a minimum of the potential function $V$.

The perspective of \cite{RT} is related to the one of \cite{LW}
which, using Nehari conditions and a cut-off operator,
simplifies the original techniques for 
gluing together mountain-pass type solutions of \cite{CZR1}, \cite{CZR2} and \cite{S}.

Our method consists in defining a Nehari-type set, $\nmu$,
by imposing that the derivative of the associated Euler-Lagrange functional
at a function $u$ should vanish when applied to the positive
and negative parts of some projections of $u$.
The idea to use these projections
is borrowed from \cite{BGH}, where they are also used, but in a different way.

We prove that the Euler-Lagrange functional associated to (\ref{one-one})
has a minimum over the set $\nmu$ using
an argument similar to the one found in \cite{CCN}. Since our
set $\nmu$ is not a manifold (see \cite[Lemma 3.1]{BW}),
one has to demonstrate, as in \cite{CSS}, that the minima are indeed critical points.
As mentioned above, in the case that $\tilde{\Omega}=\emptyset$
we recover the main result of \cite{BGH}, but with a simpler proof.

Our results are somewhat parallel to the ones 
of singular perturbation problems like in \cite{dPF}.
The large parameter $\mu$ in (\ref{one-one}) plays the role of the small parameter
$\epsilon$.
The solutions concentrate in the set $\tilde{\Omega}\cup\hat{\Omega}$
and vanish in the set $\underline{\Omega}\cup\Omega^-$ as $\mu\to+\infty$.

In \cite{AW} flow invariance properties together with
a weak splitting condition proved the existence of 
infinitely many geometrically distinct two bump solutions
of a periodic superlinear Schr\"{o}dinger equation.
The paper \cite{BCW} is concerned with the singular perturbed 
equation above. As a special case, the authors observed the existence of multiple pairs of concentrating nodal 
solutions at an isolated minimum of the potential.

There has been much interest in elliptic problems with a sign
changing weight. We refer to \cite{AdP},
\cite{AT}, \cite{BCN}, \cite{CRT}, \cite{R}, \cite{RTT}, \cite{T}  and the references therein.

For simplicity we restrict the proof to the case where
$I=J=K=1$, but it extends
to the other ones as well.
The work is organized as follows. In Section~2 
we provide estimates for minimizing sequences
on the set $\nmu$. In Section~3 we prove the existence of a minimizer
in the set $\nmu$. Finally, in Section~4 we prove that
a minimizer in the set $\nmu$ is a critical point using a
local deformation and a degree argument similar to the one in \cite{CW}.

\section{Estimates for minimizing sequences on a Nehari-type set $\nmu$}
As mentioned in the Introduction,
we consider a Lipschitz bounded domain $\Omega\subset\R^N$, $N\geq 1$, and a function
$a\in C(\bar{\Omega})$. 
We assume the set $\ap>0$ is the union of three Lipschitz components,
\begin{eqnarray}
\{x\in\Omega\::\:\ap(x)>0\}&=&\ot\cup\ohat\cup\ob,\nonumber\\
\noalign{\noindent and}
\{x\in\Omega\::\:\am(x)>0\}&=&\Omega\setminus\overline{(\ot\cup\ohat\cup\ob)}.
\label{fourteen}
\end{eqnarray}
We introduce a positive parameter $\mu$ and
consider $2<p<2^*$.

We denote by $\left\langle\ ,\ \right\rangle$ the usual inner product
on the Sobolev space $\hozo$, i.e.\
$\left\langle u ,v \right\rangle=\int\nabla u\cdot\nabla v$ for $u$, $v\in\hozo$.
When the region of integration is not specified it is understood that the integrals
are over $\Omega$. We denote by $\left\|\ \right\|$ the induced norm.
We define the spaces
\begin{eqnarray*}
\hoz(\ot)&=&\left\{u\in\hozo\::\:u=0\ \mbox{in}\ \Omega\setminus\ot\right\},\\
\hoz(\ohat)&=&\left\{u\in\hozo\::\:u=0\ \mbox{in}\ \Omega\setminus\ohat\right\},\\
\hoz(\ob)&=&\left\{u\in\hozo\::\:u=0\ \mbox{in}\ \Omega\setminus\ob\right\},
\end{eqnarray*}
which can be obtained from the spaces $H^1_0(\ot)$, $H^1_0(\ohat)$, $H^1_0(\ob)$
by extending functions as zero on
$\Omega\setminus\ot$, $\Omega\setminus\ohat$, $\Omega\setminus\ob$, respectively.

Each $u\in\hozo$ can be decomposed as 
$$u=\ut+\uh+\ub+\bu,$$
with $\ut$, $\uh$ and $\ub$ the projections of $u$ on $\hoz(\ot)$,
$\hoz(\ohat)$ and $\hoz(\ob)$, respectively. 
We recall the projections are defined by
\begin{eqnarray*}
\ut\in\hoz(\ot)\::\: \forall\,{\varphi\in\hoz(\ot)},\quad\left\langle u,\varphi\right\rangle&=&\left\langle \ut,\varphi\right\rangle,\\
\uh\in\hoz(\ohat)\::\: \forall\,{\varphi\in\hoz(\ohat)},\quad\left\langle u,\varphi\right\rangle&=&\left\langle \uh,\varphi\right\rangle,\\
\ub\in\hoz(\ob)\::\: \forall\,{\varphi\in\hoz(\ob)},\quad\left\langle u,\varphi\right\rangle&=&\left\langle \ub,\varphi\right\rangle.
\end{eqnarray*}
Clearly, these projections are orthogonal and continuous with respect to the weak topology.
The function $\bu$ is harmonic in $\ot\cup\ohat\cup\ob$.

The following is Theorem~\ref{zero} in the case when $I=J=K=1$.
\begin{Prop}\label{prop}
For every large $\mu$, there exists an $\hozo$ weak solution $u_\mu$ of
\begin{equation}
\label{one}
-\Delta u=(\ap-\mu\am)|u|^{p-2}u\quad\mbox{in}\ \Omega.
\end{equation}
Furthermore, the family $\{u_\mu\}$ has the property that, modulo a subsequence,
\begin{equation}
\label{fifteen}
u_\mu\rightharpoonup u\quad\mbox{in}\ \hozo\ \mbox{as}\ \mu\to+\infty,
\end{equation}
where 
\begin{equation}
\label{sixteen}
u=\ut+\uh,
\end{equation}
\begin{equation}
\label{seventeen}
\left\{\begin{array}{l}
-\Delta\ut=\ap|\ut|^{p-2}\ut\quad\mbox{in}\ \ot,\\
\ut^\pm\not\equiv 0,
\end{array}\right.
\end{equation}
and
\begin{equation}
\label{eighteen}
\left\{\begin{array}{l}
-\Delta\uh=\ap|\uh|^{p-2}\uh\quad\mbox{in}\ \ohat,\\
\uhp\not\equiv 0,\ \uhm\equiv 0.
\end{array}\right.
\end{equation}
\end{Prop}

The solutions of (\ref{one}) are the critical points of the $C^2$ functional $\iu\::\:
\hozo\to\R$, defined by
$$
\iu(u)=\oh\left\|u\right\|^2-\op\int(\ap-\mu\am)|u|^p.
$$

We fix a function $v$ such that
$v=\vt+\vhp$, with $\vtp$, $\vtm$, $\vh\not\equiv 0$ and
$$\iu^\prime(v)(\vtp)=\iu^\prime(v)(\vtm)=\iu^\prime(v)(\vh)=0$$
for some (and hence all) $\mu>0$.

The restriction of $\iu$ to $\hoz(\ohat)\oplus\hoz(\ob)$ is independent of
$\mu$ and has a strict local minimum at zero. 
We fix a small $\rho_0>0$ such that zero is the unique minimizer of $\iu$
in $\{u\in\hoz(\ohat)\oplus\hoz(\ob):\max\left\{\left\|\uh\right\|,\left\|\ub\right\|\}\leq\rho_0\right\}$.
For $0<\rho\leq\rho_0$, we denote by $c_\rho$ the positive constant
\begin{equation}\label{nine}
c_\rho:=\inf_{\stackrel{u\in\hoz(\ohat)\oplus\hoz(\ob)}{\rho\leq\max\{\left\|\uh\right\|,\left\|\ub\right\|\}\leq\rho_0}}\iu(u).
\end{equation}

The solutions of (\ref{one}) will be obtained by minimizing the functional $\iu$
on the following Nehari-type set, $\nmu$. Let $\rho_0$ be as
above and $R>\left\|v\right\|$.
\begin{Def}\label{def-nmu}
$\nmu$ is the set of functions $u=\ut+\uh+\ub+\bu\in\hozo$ such that
\begin{enumerate}
\item[{\rm (}${\cal{N}}_{i}${\rm )}] $\utp$, $\utm$, $\uhp\not\equiv 0$,
\item[{\rm (}${\cal{N}}_{ii}${\rm )}] $\iu^\prime(u)(\utp)=\iu^\prime(u)(\utm)=\iu^\prime(u)(\uhp)=0$,
\item[{\rm (}${\cal{N}}_{iii}${\rm )}] $\iu(u)\leq\iu(v)+1$,
\item[{\rm (}${\cal{N}}_{iv}${\rm )}] $\left\|\bu\right\|\leq\min\{\left\|\utp\right\|,\left\|\utm\right\|,\left\|\uhp\right\|\}
<\left\|\ut+\uhp\right\|\leq R$,
\item[{\rm (}${\cal{N}}_{v}${\rm )}] $\max\{\left\|\uhm\right\|,\left\|\ub\right\|\}\leq\rho_0$.
\end{enumerate}
\end{Def}
\noindent We remark that $v\in\nmu$ for all $\mu>0$.

The square of the $\hozo$ norm of $u$ is equal to the sum of the
squares of the $\hozo$ norms of the components of $u$, but the $p$-th power of the $L^p(\Omega)$
norm of $u$ does not have such a nice property. 
However, the next lemma says that this is almost the case when $\mu$ is large.
\begin{Lem}\label{L_one}
Let $\delta>0$ be given. There exists $\mu_\delta$ such that, if $\mu>\mu_\delta$,
$$\forall\, {u\in\nmu},\quad\int|\bu|^p<\delta.$$
\end{Lem}
\begin{proof}
Suppose, by contradiction, that for some $\delta>0$ there exists $\mu_n\to+\infty$ and 
$u_n\in{\cal{N}}_{\mu_n}$ with
\begin{equation}
\label{five}
\int|\bun|^p\geq\delta.
\end{equation}
As $\left\|u_n\right\|$ is bounded, we may suppose $u_n\rightharpoonup u$.
We have $\bun\rightharpoonup \bu$ and $\bu\equiv 0$ in $\Omega\setminus
(\ot\cup\ohat\cup\ob)$. Otherwise, by (\ref{fourteen}) and modulo a subsequence,
$$
\int\am|\bun|^p\geq c>0.
$$
This would contradict (${\cal{N}}_{iii}$) for sufficiently large $n$:
$$
\oh\left\|u_n\right\|^2-\op\int\ap|u_n|^p+\frac{\mu_n}{p}\int\am|\bun|^p\leq \iu(v)+1.
$$
So the function $\bu$ belongs to $\hoz(\ot)\oplus\hoz(\ohat)\oplus
\hoz(\ob)$ and is harmonic in $\ot\cup\ohat\cup\ob$. It follows that
$\bu$ must be identically equal to zero in $\Omega$. This 
contradicts (\ref{five}).
\end{proof}
Usually one may obtain a lower bound for the $\hozo$ norm of $\utp$, $\utm$ and $\uhp$ from
(${\cal{N}}_{i}$) and a condition like (${\cal{N}}_{ii}$). 
Here, in addition, we require the first inequality in (${\cal{N}}_{iv}$) to prove
\begin{Lem}\label{L_two}
There exists a constant $\kappa$, independent of $\mu$, such that
\begin{equation}
\label{six}
\forall\,{u\in\nmu},\quad\min\left\{\left\|\utp\right\|,\left\|\utm\right\|,\left\|\uhp\right\|\right\}\geq\kappa>0.
\end{equation}
\end{Lem}
\begin{proof}
Let $w$ be one of the three functions $\utp$, $-\utm$ or $\uhp$. 
Denote by $\chi$ the characteristic function of the set $\{x\in\Omega\::\: w(x)\neq 0\}$
and let $c$ be the Sobolev constant $\left(\int|v|^p\right)^{1/p}\leq c\left\|v\right\|$,
$\forall\,{v\in\hozo}$.
From $\iu^\prime(u)w=0$,
\begin{eqnarray*}
\left\|w\right\|^2&=&\int\ap|u|^{p-2}uw\leq\left\|a\right\|_\infty\left(\int{\chi}
|u|^p\right)^{\frac{p-1}{p}}\left(\int|w|^p\right)^{\frac{1}{p}}\\
&\leq&\left\|a\right\|_\infty c^p\left(\left\|\bu\right\|+\left\|w\right\|\right)^{p-1}\left\|w\right\|
\leq 2^{p-1}\left\|a\right\|_\infty c^p\left\|w\right\|^p,
\end{eqnarray*}
because of the first inequality in (${\cal{N}}_{iv}$). Since $w\not\equiv 0$, due to (${\cal{N}}_{i}$), we may take 
$$\kappa=\left(
2^{p-1}\left\|a\right\|_\infty c^p \right)^{-1/(p-2)}.$$
\end{proof}
Now we fix a $\mu$ and turn to minimizing sequences $(u_n)$ for $\iu$ restricted to $\nmu$.
Later it will be important that the limit of such a sequence has a neighborhood
whose points satisfy (${\cal{N}}_{i}$), (${\cal{N}}_{iii}$), (${\cal{N}}_{iv}$) and (${\cal{N}}_{v}$).
This follows from
\begin{Lem}\label{L_three}
Let $\overline{R}$ be fixed, $\left\|v\right\|<\overline{R}<R$, and $\delta$ be given, $0<\delta<\rho_0$.
There exists $\mu_\delta>0$ such that for every $\mu>\mu_\delta$ and every minimizing sequence $(u_n)$ for
$\iu$ restricted to $\nmu$, we have, for large $n$,
\begin{enumerate}
\item[{\rm (a)}] $\iu(u_n)\leq\iu(v)+\oh$,
\item[{\rm (b)}] $\left\|\utn+\uhpn\right\|<\overline{R}$,
\item[{\rm (c)}] $\max\{\left\|\uhmn\right\|,\left\|\ubn\right\|\}<\delta$,
\item[{\rm (d)}] $\left\|\bun\right\|<\delta$;
\end{enumerate}
also
\begin{enumerate}
\item[{\rm (e)}] $\up\int\am|\bun|^p<\delta$.
\end{enumerate}
\end{Lem}
\begin{proof}
{(a)}\/ Immediate since $(u_n)$ is minimizing and $v\in\nmu$ for all $\mu$.\\
{(b)}\/ Suppose 
\begin{equation}\label{eight}
\left\|\utn+\uhpn\right\|\geq\overline{R}
\end{equation}
for large $n$.
\begin{eqnarray*}
\iu(u_n)&=&\oh\left\|\utn+\uhpn\right\|^2+\oh\left\|\uhmn\right\|^2+\oh\left\|\ubn\right\|^2
+\oh\left\|\bun\right\|^2\\
&&-\op\int\ap|u_n|^{p-2}u_n(\utn+\uhpn)+\op\int\ap|u_n|^{p-2}u_n\uhmn\\
&&-\op\int\ap|u_n|^{p-2}u_n\ubn-\op\int\ap|u_n|^{p-2}u_n\bun+
\up\int\am|\bun|^p\\
&\geq&\left(\oh-\op\right)\left\|\utn+\uhpn\right\|^2+o(1).
\end{eqnarray*}
Here and henceforth $o(1)$ denotes a value, independent of $u\in\nmu$, that can be made arbitrarily small 
by choosing $\mu$
sufficiently large. For the proof of the last inequality we used (${\cal{N}}_{ii}$),
\begin{eqnarray*}
\oh\left\|\uhmn\right\|^2+\op\int\ap|u_n|^{p-2}u_n\uhmn&\geq&o(1)\\
\noalign{\noindent and}
\oh\left\|\ubn\right\|^2-\op\int\ap|u_n|^{p-2}u_n\ubn&\geq&o(1)
\end{eqnarray*}
(consequences of (${\cal{N}}_{v}$) and Lemma~\ref{L_one}),
$$
-\op\int\ap|u_n|^{p-2}u_n\bun=o(1)\\
$$
(consequence of (${\cal{N}}_{iv}$), (${\cal{N}}_{v}$) and Lemma~\ref{L_one}),
and 
$$
\oh\left\|\bun\right\|^2+\up\int\am|\bun|^p\geq 0.
$$
We now use (\ref{eight}) and the definition of $\overline{R}$. For sufficiently large $\mu$,
$$
\iu(u_n)\geq\left(\oh-\op\right)\overline{R}^2+o(1)>\left(\oh-\op\right)
\left\|v\right\|^2+c=\iu(v)+c,
$$
for some $c>0$.
This contradicts the fact that $(u_n)$ is minimizing.\\
{(c)}\/ Suppose $\left\|\uhmn\right\|\geq\delta$ for large $n$. As in (b), we have
\begin{eqnarray*}
\iu(u_n)
&=&\iu(u_n+\uhmn)+\oh\left\|\uhmn\right\|^2-\op\int\am|\uhmn|^p+o(1)\\
&\geq&\iu(u_n+\uhmn)+c_\delta+o(1),
\end{eqnarray*}
due to Lemma~\ref{L_one} and then (\ref{nine}).
This implies that 
$$\lim\iu(u_n)>\liminf\iu(u_n+\uhmn),$$ 
for sufficiently large $\mu$,
and contradicts the assumption that $(u_n)$ is minimizing,
because $u_n+\uhmn\in\nmu$. Similarly, one proves that $\left\|\ubn\right\|\geq\delta$ for large
$n$ leads to a contradiction, for sufficiently large $\mu$,
because $u_n-\ubn\in\nmu$.\\
{(d)}\/ Suppose $\left\|\bun\right\|\geq\delta$ for large $n$. From (${\cal{N}}_{ii}$) and Lemma~\ref{L_one},
we know
\begin{eqnarray*}
\left\|\utpn\right\|^2&=&\int\ap|\utpn|^p+o(1),\\
\left\|\utmn\right\|^2&=&\int\ap|\utmn|^p+o(1),\\
\left\|\uhpn\right\|^2&=&\int\ap|\uhpn|^p+o(1).
\end{eqnarray*}
We define $\rtn$, $\stn$ and $\thn$ by
$$
\rtn=\left(\frac{\left\|\utpn\right\|^2}{\int\ap|\utpn|^p}\right)^{\frac{1}{p-2}},\ 
\stn=\left(\frac{\left\|\utmn\right\|^2}{\int\ap|\utmn|^p}\right)^{\frac{1}{p-2}},\ 
\thn=\left(\frac{\left\|\uhpn\right\|^2}{\int\ap|\uhpn|^p}\right)^{\frac{1}{p-2}},
$$
so that $\rtn$, $\stn$, $\thn=1+o(1)$ by Lemma~\ref{L_two}, and
$$
v_n\: :=\: \rtn\utpn-\stn\utmn+\thn\uhpn-\uhmn+\ubn.
$$
Provided $\mu$ is large,
we can guarantee $v_n\in\nmu$ for large $n$ due to {(a), (b), (c)}\/ and 
Lemma~\ref{L_two}. We now obtain an upper bound for $\iu(v_n)$:
\begin{eqnarray}
\iu(v_n)&=&\iu(\utn+\uhn+\ubn)+o(1)\nonumber\\
&\leq&\iu(u_n)+o(1)\nonumber\\ &&-\left(\oh\left\|\bun\right\|^2-\op\int\ap\left(|u_n|^p
-|u_n-\bun|^p\right)
+\up\int\am|\bun|^p\right)\ \ \label{ten}\\
&\leq&\iu(u_n)+o(1)-\oh\left\|\bun\right\|^2\nonumber\\
&\leq&\iu(u_n)+o(1)-\oh\delta^2.\nonumber
\end{eqnarray}
This implies that $\liminf\iu(v_n)<\lim\iu(u_n)$
for sufficiently large $\mu$,
which is impossible.\\
{(e)}\/ Follows from inequality (\ref{ten}).
\end{proof}
\section{Existence of a minimizer in $\nmu$}

For each $u\in\nmu$, we consider the 3-dimensional manifold with boundary
in $\hozo$ parametrized on $[0,2]^3$ by
\begin{equation}\label{twelve}
\varsigma(\rt,\st,\that)=\rt\utp-\st\utm+\that\uhp-\uhm+\ub+\bu.
\end{equation}
We call $f$ the function $\iu\circ\varsigma$, so that 
\begin{eqnarray*}
f(\rt,\st,\that)&=&\frac{\rt^2}{2}\left\|\utp\right\|^2+
\frac{\st^2}{2}\left\|\utm\right\|^2+\frac{\that^2}{2}\left\|\uhp\right\|^2+K\\
&&-\op\int\ap|\rt\utp+\bu|^p-\op\int\ap|\bu-\st\utm|^p-\op\int\ap|\that\uhp+\bu|^p,
\end{eqnarray*}
with
\begin{eqnarray*}
K&=&\oh\left\|\uhm\right\|^2+\oh\left\|\ub\right\|^2+\oh\left\|\bu\right\|^2\\
&&-\op\int\ap|\bu-\uhm|^p-\op\int\ap|\ub+\bu|^p+\up\int\am|\bu|^p.
\end{eqnarray*}
Two properties of $f$ are immediate, namely
$f(1,1,1)=\iu(u)$ and $\nabla f(1,1,1)=0$ by (${\cal{N}}_{ii}$).
The critical point $(1,1,1)$ is characterized in 
\begin{Lem}\label{L_four}
For $\mu$ sufficiently large, independent of $u\in\nmu$, the point $(1,1,1)$
is an absolute maximum of $f$. Furthermore, if 
$$
|(\rt,\st,\that)-(1,1,1)|\geq\theta>0,
$$
then
\begin{equation}
\label{eleven}
f(\rt,\st,\that)\leq f(1,1,1)-d_\theta.
\end{equation}
The constant $d_\theta>0$ may be chosen independent of $u$ and $\mu$.
\end{Lem}
\begin{proof}
We define an auxiliary function $g:\left[0,2\right]^3\to\R$  by
\begin{eqnarray*}
g(\rt,\st,\that)&:=&\left(\frac{\rt^2}{2}-\frac{\rt^p}{p}\right)\left\|\utp\right\|^2
+\left(\frac{\st^2}{2}-\frac{\st^p}{p}\right)\left\|\utm\right\|^2\\ 
&&+\left(\frac{\that^2}{2}-\frac{\that^p}{p}\right)\left\|\uhp\right\|^2+K,
\end{eqnarray*}
which satisfies $\nabla g(1,1,1)=0$ and $$D^2g(1,1,1)=-(p-2)\,\mbox{diag}\left\{\left\|\utp\right\|^2,
\left\|\utm\right\|^2,\left\|\uhp\right\|^2\right\}\leq -(p-2)\kappa I,$$
where $\kappa$ was defined in Lemma~\ref{L_two}. 
One easily checks that in a small neighborhood of $(1,1,1)$ the second
derivative $D^2 g$ is below a negative definite matrix
which is independent of $u\in\nmu$.
We also have that, for any derivative $D^\alpha$ with $|\alpha|\leq 2$,
\begin{equation}
\label{twenty}
|D^\alpha f-D^\alpha g|=o(1),
\end{equation}
by Lemma~\ref{L_one}; notice that the right-hand-side is uniform in $u$ and $\mu$.
Thus, by (\ref{twenty}) with $|\alpha|=2$, $f$ has a strict local maximum at $(1,1,1)$.
We take $\alpha=0$ to conclude this maximum is absolute. Of course, the previous two statements
hold provided $\mu$ is sufficiently large.
\end{proof}

Let $\mu$ be fixed and $(u_n)$ be a minimizing sequence for $\iu$ restricted to $\nmu$.
Since $\nmu$ is bounded in $\hozo$, we may assume 
$$u_n\rightharpoonup u\quad \mbox{in}\ \hozo.$$
\begin{Lem}\label{L_five} If $\mu$ is sufficiently large,
the function $u$ belongs to $\nmu$. Therefore (by the lower semi-continuity of the norm) 
the function $u$ is a minimizer of 
$\iu$ restricted to $\nmu$.
\end{Lem}
\begin{proof}
We may assume $\utpn\rightharpoonup\utp$, $\utmn\rightharpoonup\utm$, $\uhpn\rightharpoonup\uhp$
in $\hozo$, since $w_n\rightharpoonup w$ in $\hozo$ implies a subsequence of $w_n$
converges pointwise a.e.\ to $w$. From (${\cal{N}}_{ii}$) and Lemma~\ref{L_two},
$$
\min\left\{
\int\ap|u|^{p-2}u\utp,\ -\int\ap|u|^{p-2}u\utm,\
\int\ap|u|^{p-2}u\uhp\right\}\geq\kappa.$$ These three integrals are also bounded above 
by a constant independent of $\mu$ because $\nmu$ is bounded.
It follows from Lemma~\ref{L_one} that the integrals 
$$
\int\ap|\utp|^p,\ \ \int\ap|\utm|^p,\ \ \int\ap|\uhp|^p
$$
are bounded below by a positive constant independent of $\mu$. The Sobolev inequality
now implies that the norms
$$
\left\|\utp\right\|,\ \ \left\|\utm\right\|,\ \ \left\|\uhp\right\| 
$$
are bounded below by a positive constant independent of $\mu$.
From the lower-semicontinuity of the norm,
\begin{equation}\label{thirteen}
\left\|\utp\right\|\leq\liminf\left\|\utpn\right\|,\ \ 
\left\|\utm\right\|\leq\liminf\left\|\utmn\right\|,\ \ 
\left\|\uhp\right\|\leq\liminf\left\|\uhpn\right\|.
\end{equation}
We wish to prove that equalities hold. Otherwise, choose $(\rt,\st,\that)$,
defined by 
\begin{eqnarray*}
\rt&=&\left(\frac{\left\|\utp\right\|^2}{\int\ap|u|^{p-2}u\utp}\right)^{\frac{1}{p-2}},\ 
\st=\left(\frac{\left\|\utm\right\|^2}{-\int\ap|u|^{p-2}u\utm}\right)^{\frac{1}{p-2}},\\ 
\that&=&\left(\frac{\left\|\uhp\right\|^2}{\int\ap|u|^{p-2}u\uhp}\right)^{\frac{1}{p-2}},
\end{eqnarray*}
so that the function
$$
w:=\rt\utp-\st\utm+\that\uhp-\uhm+\ub+\bu
$$
satisfies (${\cal{N}}_{ii}$). By (\ref{thirteen}), the strong convergence in $L^p(\Omega)$,
and what we have just seen,
$$(\rt,\st,\that)\in\,\,[c,1]^3\setminus\{(1,1,1)\},$$
for some $c>0$ independent of $\mu$.
The function $w$ clearly satisfies (${\cal{N}}_{i}$) and (${\cal{N}}_{v}$).
Lemma~\ref{L_one} guarantees that 
(${\cal{N}}_{iv}$) is satisfied for sufficiently large $\mu$.
Consider the estimate
\begin{eqnarray*}
\noalign{$\displaystyle\iu(\rt\utp-\st\utm+\that\uhp-\uhm+\ub+\bu)$}
\qquad\qquad&<&\liminf\iu(\rt\utpn-\st\utmn+\that\uhpn-\uhmn+\ubn+\bun)\\
&\leq&\lim\iu(u_n),
\end{eqnarray*}
where the last inequality is due to Lemma~\ref{L_four}.
It shows that $w$ satisfies (${\cal{N}}_{iii}$).
Therefore $w\in\nmu$ and $\iu(w)<\lim\iu(u_n)$. This is a contradiction.
We have established that equality holds in 
all three of (\ref{thirteen}).
Therefore $u\in\nmu$ for large $\mu$.
\end{proof}
\section{A minimizer in $\nmu$ is a critical point}
In the previous section we obtained a minimizer $u$ of $\iu$ on $\nmu$.
We will now prove that this minimizer is indeed a critical point of $\iu$.
This will be done by using a deformation argument on the manifold
introduced above. Let $\sigma$ be the restriction to the interval
$[1/2,2]^3$ of the $\varsigma$ corresponding to the minimizer $u$.
Recall $\varsigma$ was defined in (\ref{twelve}).
We define a negative gradient flow in a neighborhood of $u$
in the following way.
Let $B_\rho(u):=\{w\in\hozo\::\:\left\|w-u\right\|<\rho\}$,
where $\rho$ is chosen small enough so that 
\begin{equation}\label{seven}
\sigma(\rt,\st,\that)\in B_\rho(u)
\ \Rightarrow \ \oh<\rt,\st,\that<2
\end{equation}
and $w\in B_\rho(u)$ implies that $w$ satisfies
(${\cal{N}}_{i}$), (${\cal{N}}_{iii}$), (${\cal{N}}_{iv}$) and (${\cal{N}}_{v}$),
for sufficiently large $\mu$. Such a $\rho$ exists because the function $u$
satisfies (\ref{six}) and (a), (b), (c) and (d) of Lemma~\ref{L_three}.
Let $\varphi$ be a Lipschitz function, $\varphi:\hozo\to[0,1]$,
such that $\varphi=1$ on $B_{\rho/2}(u)$ and $\varphi=0$ on the complement
of $B_\rho(u)$. Consider the Cauchy problem
\begin{equation}\label{twenty-one}
\left\{\begin{array}{l}
\displaystyle\frac{d\eta}{d\tau}=-\varphi(\eta)\nabla I_\mu(\eta),\\
\vspace{-4mm}\\
\eta(0)=w,\end{array}\right.
\end{equation}
whose solution we denote by $\eta(\tau;w)$. For $\tau\geq 0$, let
$$
\sigma_\tau(\rt,\st,\that)=\eta(\tau;\sigma(\rt,\st,\that)).
$$
\begin{Lem}\label{L_six}
The set $\sigma_\tau\left(\left[1/2,2\right]^3\right)$ intersects $\nmu$ in an nonempty set.
\end{Lem}
\begin{proof}
Consider the maps $\tilde{\phi}^\pm$, $\hat{\phi}$, $\tilde{\psi}^\pm$, $\hat{\psi}$ from
$\left\{w\in\hozo\::\:\tilde{w}^\pm\not\equiv 0, \whp\not\equiv 0\right\}$ to $\R$, defined by
$$\begin{array}{ll}\displaystyle
\tilde{\phi}^\pm(w)=\frac{\pm\int\ap|w|^{p-2}w\tilde{w}^\pm}{\left\|\tilde{w}^\pm\right\|^2},&
\displaystyle\hat{\phi}(w)=\frac{\int\ap|w|^{p-2}w\whp}{\left\|\whp\right\|^2},\\
\vspace{-2mm}&\\
\displaystyle\tilde{\psi}^\pm(w)=\frac{\int\ap|\tilde{w}^\pm|^p}{\left\|\tilde{w}^\pm\right\|^2},&
\displaystyle\hat{\psi}(w)=\frac{\int\ap|\whp|^p}{\left\|\whp\right\|^2}.
\end{array}
$$
These maps are well defined on $\sigma_\tau\left([1/2,2]^3\right)$, because if $w\in B_\rho(u)$,
then $w$ satisfies (${\cal{N}}_{i}$). We finally define
$$
\Phi_\tau:=\left(\tilde{\phi}^+,\tilde{\phi}^-,\hat{\phi}\right)\circ\sigma_\tau
$$
and
$$
\Psi:=\left(\tilde{\psi}^+,\tilde{\psi}^-,\hat{\psi}\right)\circ\sigma,
$$
from $\left([1/2,2]^3\right)$ to $\R^3$.
Since $\int|\bu|^p=o(1)$ uniformly in $u$ and $\mu$ and the value of $\kappa$
in Lemma~\ref{L_two} is independent of $\mu$,
\begin{eqnarray}
\Psi(\rt,\st,\that)&=&\left(\rt^{p-2}\tilde{\psi}^+(u),\st^{p-2}\tilde{\psi}^-(u),
\that^{p-2}\hat{\psi}(u)\right)\nonumber\\
&=&\left((1+o(1))\rt^{p-2},(1+o(1))\st^{p-2},(1+o(1))\that^{p-2}\right),\label{nineteen}
\end{eqnarray}
with the last three $o(1)$ independent of $u$ and $\mu$. As a consequence,
$$
\mbox{dist}\left(\Psi\left(\partial[1/2,2]^3\right), (1,1,1)\right)\geq c>0,
$$
the constant $c$ being independent of $u$ and $\mu$.
We deduce from (\ref{nineteen}) that for large $\mu$,
$$
\mbox{deg}\,\left(\Psi,\left[1/2,2\right]^3,(1,1,1)\right)=1.
$$
Notice that condition (\ref{seven}) and the definition of the flux (\ref{twenty-one}) guarantee
$$
\left.\Phi_\tau\right|_{\partial\left[1/2,2\right]^3}=
\left.\Phi_0\right|_{\partial\left[1/2,2\right]^3}=
\left.\Psi\right|_{\partial\left[1/2,2\right]^3}+o(1)
$$
and therefore
$$
\mbox{deg}\,\left(\Phi_\tau,\left[1/2,2\right]^3,(1,1,1)\right)=1.
$$
for $\mu$ large enough. This proves that 
$$\sigma_\tau\left(\left[1/2,2\right]^3\right)\cap\nmu\neq\emptyset.$$
\end{proof}

We are ready to give the\\
\begin{altproof}{\/\/ {\rm Proposition~\ref{prop}}}
Let $\mu$ be large and $u_\mu$ be a minimizer of $\iu$ restricted to $\nmu$.
The existence of such a $u_\mu$ was proven in Lemma~\ref{L_five}.
Suppose that $\iu^\prime(u_\mu)\neq 0$. By Lemma~\ref{L_four}, with $u=u_\mu$,
$\max\iu\circ\sigma\left(\left[1/2,2\right]^3\right)=\iu(u_\mu)$, and so for any small $\tau>0$,
$$
\max\iu\circ\sigma_\tau\left(\left[1/2,2\right]^3\right)<\iu(u_\mu).
$$
This 
contradicts Lemma~\ref{L_six}. So $\iu^\prime(u_\mu)=0$, and the minimizer of
$\iu$ on $\nmu$ is a weak solution of (\ref{one}).

Consider now $u$ as in (\ref{fifteen}). 
Properties (\ref{sixteen}), (\ref{seventeen}) and (\ref{eighteen})
follow from Lemma~\ref{L_two} and Lemma~\ref{L_three} (c), (d), as
$$
\min\left\{
\int\ap|u_\mu|^{p-2}u_\mu\tilde{u}_\mu^+,\ 
-\int\ap|u_\mu|^{p-2}u_\mu\tilde{u}_\mu^-,\ 
\int\ap|u_\mu|^{p-2}u_\mu\hat{u}_\mu^+
\right\}\geq\kappa.
$$
\end{altproof}

Theorem~\ref{zero} can be proved as Proposition~\ref{prop} with obvious adaptations.

\end{document}